\documentclass[a4paper,12pt,reqno]{amsart}

\usepackage[T2A]{fontenc}
\usepackage[cp1251]{inputenc}
\usepackage[english]{babel}
\usepackage{amsmath,amsthm,amsfonts,amssymb}
\usepackage{enumerate}

\theoremstyle{plain}
\newtheorem{theorem}{Theorem}
\newtheorem{lemma}{Lemma}
\newtheorem{corollary}{Corollary}
\newtheorem*{corollary*}{Наслідок}
\newtheorem{proposition}{Proposition}
\theoremstyle{definition}

\theoremstyle{remark}
\newtheorem{remark}{Remark}
\newtheorem*{remark*}{Remark}


\uccode`ґ=`Ґ \uccode`Ґ=`Ґ \lccode`Ґ=`ґ \lccode`ґ=`ґ
\uccode`є=`Є \uccode`Є=`Є \lccode`Є=`є \lccode`є=`є
\uccode`і=`І \uccode`І=`І \lccode`І=`і \lccode`і=`і
\uccode`ї=`Ї \uccode`Ї=`Ї \lccode`Ї=`ї \lccode`ї=`ї

\begin{document}

\title[On phenomena connected to the  second Ostrogradsky expansion]{On metric, dynamical, probabilistic\\ and fractal  phenomena connected with \\ the second Ostrogradsky expansion}
\author[S.Albeverio, I.Pratsiovyta, G.Torbin]
{Sergio Albeverio$^{1,2,3,4,5}$, Iryna ~Pratsiovyta$^{6}$, \\ Grygoriy Torbin$^{7,8}$}

\begin{abstract}
 We consider  the second Ostrogradsky expansion from the number theory, probability theory, dynamical systems   and fractal geometry points of view, and  establish several new  phenomena connected with this expansion.

First of all we prove the singularity of the random second Ostrogradsky expansion.

Secondly we study properties of the symbolic dynamical system generated by the natural one-sided shift-transformation $T$ on the second Ostrogradsky expansion. It is shown, in particular, that there are no probability measures which are simultaneously invariant and ergodic (w.r.t. $T$) and  absolutely continuous (w.r.t. Lebesgue measure). So, the classical  ergodic approach to the development of the metric theory is not applicable for the  second Ostrogradsky expansion.

We develop instead the metric and dimensional theories for this expansion
using probabilistic methods. We show, in particular,  that  for Lebesgue almost all real numbers any digit $i$ from the  alphabet $A= \mathbb{N} $ appears  only finitely many times  in the difference-version of the second Ostrogradsky expansion, and the set of all reals with bounded digits of this expansion is of zero Hausdorff dimension.

Finally, we compare metric, probabilistic and dimensional theories for the second Ostrogradsky expansions with the corresponding theories for the Ostrogradsky-Pierce expansion and for the continued fractions expansion.

\end{abstract}

\maketitle

$^1$~Institut f\"ur Angewandte Mathematik, Universit\"at Bonn,
Endenicher Allee 60, D-53115 Bonn (Germany); $^2$HCM and ~SFB 611, Bonn; $^3$ BiBoS,
Bielefeld--Bonn; $^4$~CERFIM, Locarno and Acc. Arch., USI
(Switzerland); $^5$~IZKS, Bonn; E-mail:
albeverio@uni-bonn.de

$^{6}$~National Pedagogical University, Pyrogova str. 9, 01030 Kyiv
(Ukraine); E-mail: lightsoul2008@gmail.com

$^7$~National Pedagogical University, Pyrogova str. 9, 01030 Kyiv
(Ukraine) $^{8}$~Institute for Mathematics of NASU,
Tereshchenkivs'ka str. 3, 01601 Kyiv (Ukraine); E-mail:
torbin@iam.uni-bonn.de (corresponding author)




\section{Introduction}
The presented paper is devoted to the investigation of  the expansions
of real numbers in the  Ostrogradsky series (they were
introduced by M.~V.~Ostrogradsky, a well known ukrainian
mathematician  who lived from 1801 to 1862).

The expansion of $x$ of the form:
\begin{equation}\label{?}
x=\frac{1}{q_1}-\frac{1}{q_1q_2}+\dots
+\frac{(-1)^{n-1}}{q_1q_2\dots q_n}+\dotsb,
\end{equation}
where $q_n$ are positive integers and $q_{n+1}>q_n$ for all $n$,
is said to be the expansion of $x$ in the first Ostrogradsky
series. The expansion of $x$ of the form:
\begin{equation}\label{??}
x=\frac{1}{q_1}-\frac{1}{q_2}+\dots
+\frac{(-1)^{n-1}}{q_n}+\dotsb,
\end{equation}
where $q_n$ are positive integers and $q_{n+1}\geq q_n(q_n+1)$ for
all $n$, is said to be the expansion of $x$ in the second
Ostrogradsky series. Each irrational number has a unique expansion
of the form~\eqref{?} or~\eqref{??}. Rational numbers have two
finite different representations of the above form (see, e.g.,
\cite{Rem51}).

Shortly before his death, M.~Ostrogradsky has proposed two
algorithms for the representation of real numbers via  alternating
series of the form \eqref{?} and \eqref{??}, but he did not publish
any papers on this problems. Short Ostrogradsky's remarks concerning
the above representations have been found by E.~Ya.~Remez
\cite{Rem51} in the hand-written fund of the Academy of Sciences of
USSR. E.~Ya.~Remez has pointed out some similarities between the
Ostrogradsky series and continued fractions. He also paid a great
attention to the applications of the Ostrogradsky series for the
numerical methods for solving algebraic equations. In the editorial
comments to the book \cite{Khi61/63} B.~Gnedenko has pointed out that
there are no fundamental investigations of properties of the above
mentioned representations.
Since that time there were a lot of publications devoted to the first Ostrogrdsky series (see, e.g.,  \cite{ABPT,ABPT2} and references and historical notes therein). Unfortunately the metric theory of the second Ostrogradsky is still not well developed.

The second Ostrogradsky  series converges rather quickly, giving a very good approximation of irrational numbers by rationals, which are partial sums of the above series.

The representation of a real number  $x$ in the form (\ref{??})
  is said to be the {\it $O^{2}$-expansion }of $x$ and will be denoted by $O^{2}(q_{1}(x),\ldots,q_{n}(x), ...)$.
For a given set of $k$ initial    symbols,  the (k+1)-th symbol of the  $O^{2}$-expansion can not take the values  $1,2,3,4,5,..., q_k(q_k+1) -1$, which makes these   $O^{2}$-symbols dependent.

  Let
\[
d_1 = q_1\quad \text{and}\quad d_{k+1}= q_{k+1}- q_k(q_k+1)+1 \quad \forall k \in N.
\]
 Then the previous series can be rewritten in the form
\begin{equation}\label{eq:bo2series}
x= \sum_k\frac{(-1)^{k+1}}{q_{k-1}(x)(q_{k-1}(x)+ 1) -1 + d_{k}(x)}=:\bar{O}^{2}(d_1(x),d_2(x),\dots, d_k(x),\dots).
\end{equation}
 Expression ~\eqref{eq:bo2series} is said to be the  $\bar{O}^{2}$-expansion (the second Ostrogradsky expansion with independent increments). The number  $d_k=d_k(x)$ is said to be the  $k$-th  $\bar{O}^{2}$-symbol of  ~$x$. In the  $\bar{O}^{2}$-expansion any symbol (independently of all previous ones) can  be chosen from the whole set of positive integers.

The main aims of the present paper are:

1) to develop  ergodic, metric and dimensional theories of the $\bar{O}^{2}$ - expansion for real numbers and compare these theories with the  corresponding theories for the continued fraction expansion and for the Ostrogradsky-Pierce expansion;

2) to study properties of the symbolic dynamical system generated by the one-sided shift transformation on the $\bar{O}^{2}$-expansion:
$$
 \forall ~ x = \bar{O}^{2}(d_1(x), d_2(x), \dots, d_n(x), \dots)~~ \in [0,1],
 $$
  $$T(x) = T(\bar{O}^{2}(d_1(x), d_2(x),\dots, d_n(x),\dots))= \bar{O}^{2}(d_2(x), d_3(x),\dots, d_n(x),\dots);
$$
3) to study  the distributions of random variables
 $$\eta =\bar{O}^{2} ( \eta_1, \eta_2, ..., \eta_n, ... ),
$$
whose $\bar{O}^{2}$-symbols $\eta_k$ are  independent identically distributed random variables taking the values $1$, $2$,~$\dots$, $m$,~$\dots$ with probabilities $p_{1}$, $p_{2}$,~$\dots$, $p_{m}$,~$\dots$
respectively, $ p_{m}\geq0, \quad
\sum\limits_{m=1}^\infty {p_{m}}=1.
$

\section{Properties of  cylinders of the $O^{2}$- and the $\bar{O}^{2}$-expansions }


Let $(c_{1},c_{2},\ldots,c_{n})$ be a given set of natural numbers.
 The set $$\Delta_{c_{1}c_{2}\ldots
c_{n}}^{O^{2}}=\{x:x=O^{2}(q_{1},q_{2},\ldots,q_{n},\ldots),q_{i}=c_{i},i=\overline{1,n},q_{n+j}\in
N\}$$ is said to be the cylinder of rank $n$ with the base $c_{1}c_{2}\ldots c_{n}$

Let us mention some properties of the $O^{2}$-cylinders.
\begin{enumerate}
\item[1.] $ \Delta_{c_{1}\ldots c_{n}}^{O^{2}}=\bigcup\limits_{i=c_{n}(c_{n}+1)}^{\infty}\Delta_{c_{1}\ldots c_{n}i}^{O^{2}}.$


\item[2.] $\inf\Delta_{c_{1}\ldots 2m-1}^{O^{2}}=\sum\limits_{k=1}^{2m-1}\frac{(-1)^{k-1}}{c_{k}}-\frac{1}{c_{2m-1}(c_{2m-1}+1)}=O^{2}(c_{1}, \ldots, c_{2m-1},c_{2m-1}(c_{2m-1}+1))=$
    $O^{2}(c_{1}, \ldots, c_{2m-2},c_{2m-1}+1)\in \Delta_{c_{1}\ldots c_{2m-1}}^{O^{2}};$
    \\$\sup\Delta_{c_{1}\ldots 2m-1}^{O^{2}}=\sum\limits_{k=1}^{2m-1}\frac{(-1)^{k-1}}{c_{k}}=O^{2}(c_{1}, \ldots, c_{2m-1})\in \Delta_{c_{1}\ldots c_{2m-1}}^{O^{2}};$
    \\$ \inf\Delta_{c_{1}\ldots 2m}^{O^{2}}=\sum\limits_{k=1}^{2m}\frac{(-1)^{k-1}}{c_{k}}=O^{2}(c_{1}, \ldots, c_{2m})\in \Delta_{c_{1}\ldots c_{2m}}^{O^{2}};$
    \\$\sup\Delta_{c_{1}\ldots 2m}^{O^{2}}=\sum\limits_{k=1}^{2m}\frac{(-1)^{k-1}}{c_{k}}+\frac{1}{c_{2m}(c_{2m}+1)}=O^{2}(c_{1}, \ldots, c_{2m},c_{2m}(c_{2m}+1))=$
    $O^{2}(c_{1},\ldots,c_{2m-1},c_{2m}+1)\in \Delta_{c_{1}\ldots c_{2m}}^{O^{2}}.$
\item[3.] $\sup \Delta^{O^{2}}_{c_{1}\ldots c_{2m-1}i}=\inf \Delta^{O^{2}}_{c_{1}\ldots c_{2m-1}(i+1)};$
    \\$\inf \Delta^{O^{2}}_{c_{1}\ldots c_{2m}i}=\sup \Delta^{O^{2}}_{c_{1}\ldots c_{2m}(i+1)}.$
\end{enumerate}

\begin{lemma}
The cylinder  $\Delta_{c_{1}\ldots c_{n}}^{O_{2}}$ is a closed interval $[a,b],$ where
$$a=\left\{ \begin{array}{lll}
O_{2}(c_{1},\ldots,c_{n},c_{n}(c_{n}+1)), & \mbox{if  n is odd} &, \\
O_{2}(c_{1},\ldots,c_{n}), & \mbox{if n is even} &;
\end{array} \right. $$
$$b=\left\{ \begin{array}{lll}
O_{2}(c_{1},\ldots,c_{n}), & \mbox{if n is odd} &, \\
O_{2}(c_{1},\ldots,c_{n}(c_{n}+1)), & \mbox{if n is even} &.
\end{array} \right. $$
\end{lemma}
\begin{proof}
It is clear that  $\Delta_{c_{1}\ldots c_{n}}^{O_{2}}\subset [a,b].$ Let us prove that $[a,b]\subset\Delta_{c_{1}\ldots c_{n}}^{O_{2}}. $
Since $a,b\in \Delta_{c_{1}\ldots c_{n}}^{O_{2}},$ it is enough to show that any $x\in(a,b)$belongs to  $\Delta_{c_{1}\ldots c_{n}}^{O_{2}}.$

If $a<x<b,$ then $$x=a+x_{1},~\mbox{where}~a<x_{1}<b-a=\frac{1}{c_{n}(c_{n}+1)}.$$ The Ostrogradsky-Remez theorem states that  $x_{1}$ can be decomposed into the second Ostrogradsky series by using the following algorithm:
$$\left\{ \begin{array}{l}
 1=q_{1}x+\beta_{1}~~~ (0 \leq \beta_{1} <x), \\
 q_{1}=q_{2}\beta_{1}+\beta_{2}~~~ (0 \leq \beta_{2} < \beta_{1}), \\
 q_{2}q_{1}=q_{3}\beta_{2}+\beta_{3}~~~ (0 \leq \beta_{3} < \beta_{2}), \\
 .............................. \\
 q_{k}...q_{2}q_{1}=q_{k+1}\beta_{k}+\beta_{k+1}~~~ (0 \leq  \beta_{k+1}
< \beta_{k}), \\
 ..............................
\end{array} \right.$$ i.e.,

$$x_{1}=\frac{1}{q_{1}'}-\frac{1}{q_{2}'}+\frac{1}{q_{3}'}-\ldots+(-1)^{k-1}\frac{1}{q_{k}'}+\ldots . $$
Let us show that  $q_{1}'\geq c_{n}(c_{n}+1).$ Since
$$\frac{1}{q_{1}'+1}\leq \frac{1}{q_{1}'}-\frac{1}{q_{2}'}\leq x_{1}\leq\frac{1}{q_{1}'}, $$ we get
$$\frac{1}{q_{1}'+1}\leq x_{1}<\frac{1}{c_{n}(c_{n}+1)},$$
and hence, $c_{n}(c_{n}+1)<q_{1}'+1$ and $c_{n}(c_{n}+1)\leq q_{1}'.$

So, the number $x$ has the following $O_{2}$-expansion:
$x=O_{2}(c_{1},\ldots,c_{n},q_{1}',q_{2}',\ldots),$
and, therefore, $x\in \Delta_{c_{1}\ldots c_{n}}^{O_{2}}, $ which proves the lemma.
\end{proof}

\begin{corollary}For the length of the cylinder of rank $n$ the following relations hold: $$ |\Delta_{c_{1}\ldots \ldots
c_{n}}^{O^{2}}|=\frac{1}{c_{n}(c_{n}+1)}\rightarrow
0~~(n\rightarrow\infty);$$
\end{corollary}

\begin{corollary}
The length of a cylinder depends only on the last digit of the base: $$|\Delta_{c_{1}\ldots
c_{n}i}^{O^{2}}|=\frac{1}{i(i+1)}=|\Delta_{s_{1}\ldots s_{m}i}^{O^{2}}|,~~\forall i\geq \max\{c_{n}(c_{n}+1), s_{m}(s_{m}+1)\}.$$
\end{corollary}

Any cylinder of the $O^{2}$-expansion can be rewritten in terms of  the $\bar{O}^{2}$-expansion:
$$\Delta_{c_{1}\ldots c_{n}}^{O^{2}}\equiv \bar{\Delta}_{a_{1}\ldots
\ldots a_{n}}^{O^{2}},$$
 where $a_{1}=c_{1},~a_{k}=c_{k}+1-c_{k-1}(c_{k-1}+1),~1<k<n.$

Properties of   $O^{2}$-cylinders generate properties of the $\bar{O}^{2}$-cylinders. Let us mention only some of these properties.


\begin{lemma}\label{lemma relation of cyl}
$$\frac{|\Delta_{a_{1}\ldots a_{n}i}^{\bar{O^{2}}}|}{|\Delta_{a_{1}\ldots a_{n}}^{\bar{O^{2}}}|}=
\frac{|\Delta_{c_{1}\ldots c_{n}(c_{n}(c_{n}+1)-1+i)}^{O^{2}}|}{|\Delta_{c_{1}\ldots c_{n}}^{O^{2}}|}\leq \frac{1}{c_{n}(c_{n}+1)+i}<\frac{1}{c_{n}^{2}}.$$
\end{lemma}
\begin{proof}
$$\frac{|{\Delta}_{a_{1}\ldots a_{n}i}^{\bar{O^{2}}}|}{|{\Delta}_{a_{1}\ldots a_{n}}^{\bar{O^{2}}}|}=
\frac{|\Delta_{c_{1}\ldots c_{n}(c_{n}(c_{n}+1)-1+i)}^{O^{2}}|}{|\Delta_{c_{1}\ldots c_{n}}^{O^{2}}|}=\frac{c_{n}(c_{n}+1)}{(c_{n}(c_{n}+1)+i-1)(c_{n}(c_{n}+1)+i)}=$$
$$=\frac{1}{\left(1+\frac{i-1}{c_{n}(c_{n}+1)}\right)
\left(c_{n}(c_{n}+1)+i\right)}\leq\frac{1}{c_{n}(c_{n}+1)+i}<\frac{1}{c_{n}^{2}}.$$
\end{proof}

\begin{remark}
For the continued fractions
$$\frac{1}{3i^{2}}\leq \frac{|\Delta_{a_{1}\ldots a_{n}i}^{c.f.}|}{|\Delta_{a_{1}\ldots a_{n}}^{c.f.}|}\leq \frac{1}{i^{2}},$$
independently of the  previous digits.

\end{remark}

\begin{lemma} \label{Lemma ocinka vidnoshennya cylindriv}
For any sequence of positive integers $a_1, a_2,..., a_k$ the following inequalities hold:
$$\frac{|{\Delta}_{a_{1}a_{2}\ldots a_{k}
j}^{\bar{O}^{2}}|}{|{\Delta}_{a_{1}a_{2}\ldots
a_{k}}^{\bar{O}^{2}}|}<\frac{1}{2^{2^{k-1}}}, ~\forall j\in N; $$
$$\frac{|\Delta_{b a_{2}\ldots a_{k}j}^{\bar{O}^{2}}|}
{|\Delta_{b a_{2}\ldots a_{k}}^{\bar{O}^{2}}|}<\frac{1}{b^{2^{k}}},~\forall b\in N, ~\forall j\in N.$$
\end{lemma}
\begin{proof}
From lemma \ref{lemma relation of cyl} and from the fact that $c_{k}>c_{k-1}^{2}>c_{k-2}^{2^{2}}>c_{k-3}^{2^{3}}>\ldots>c_{k-(k-2)}^{2^{k-2}}>c_{2}^{2^{k-2}}>2^{2^{k-2}}$ it follows that
$$\frac{|{\Delta}_{a_{1}a_{2}\ldots a_{k}
j}^{\bar{O}^{2}}|}{|{\Delta}_{a_{1}a_{2}\ldots
a_{k}}^{\bar{O}^{2}}|}<\frac{1}{c_{k}^{2}}<\frac{1}{(2^{2^{k-2}})^{2}}=\frac{1}{2^{2^{k-1}}}.$$

Similarly, taking into account the fact that $c_{k}>c_{2}^{k-2}$ and  $c_{2}\geq b(b+1)>b^{2},$ we get
$$\frac{|\Delta_{b a_{2}\ldots a_{k}j}^{\bar{O}^{2}}|}
{|\Delta_{b a_{2}\ldots a_{k}}^{\bar{O}^{2}}|}<\frac{1}{c_{k}^{2}}<\frac{1}{c_{2}^{k-1}}<\frac{1}{(b^{2^{k-1}})^{2}}=\frac{1}{b^{^{2^{k}}}},~\forall b\in N.$$
\end{proof}

\section{Properties of  symbolic dynamics related to the $\bar{O}^{2}$-expansion}

Let us consider the dynamical system generated by the one-sided shift transformation  $T$ on the  $\bar{O}^2$-expansion:
$$
 \forall ~ x = \bar{O}^2(d_1(x), d_2(x), \dots, d_n(x), \dots)~~ \in [0,1],
 $$
  $$T(x) =  \bar{O}^2(d_2(x), d_3(x),\dots, d_n(x),\dots).
$$

 Recall that a set  $A$ is said to be invariant w.r.t. a measurable transformation  $T,$ if $A=T^{-1}A.$
  A measure $\mu$ is said to be ergodic w.r.t. a transformation  $T,$ if any invariant set  $A\in\mathfrak{B}$ is either of full or of zero measure $\mu$. A measure $\mu$ is said to be invariant w.r.t. a transformation  $T,$ if for any set  $E\in\mathfrak{B}$ one has
  $\mu(T^{-1}E)=\mu(E)$.

 It is well known that the development of  metric and ergodic theories of some expansion for reals can be essentially simplified if one can find a measure which is invariant and ergodic w.r.t. one-sided shift transformation on the corresponding expansion and absolutely continuous w.r.t. Lebesgue measure.  For instance, having the Gauss measure (i.e., the measure with the density $f(x)= \frac{1}{\ln 2} \frac{1}{1+x}$ on the unit interval) as invariant and ergodic measure w.r.t. the transformation $T(x)=\frac{1}{x} (mod 1)$, one can easily derive main metric and ergodic properties of the continued fraction (c.f.)-expansion. In this section we shall try to find  the $\bar{O}^{2}$-analogue of the Gauss measure.

Let  $N_{i}(x,k)$ be the number of the symbols $ "i" $ among the first $k$ symbols of  the $\bar{O}^{2}$-expansion of  $x$.
If the limit $\lim\limits_{k \to \infty} \frac{N_{i}(x,k)}{k}$  exists, then it is said to be the asymptotic frequency of the digit (symbol) "i" in the $\bar{O}^{2}$-expansion of the real number $x$. We shall use the notation  $\nu_i (x, \bar{O}^{2})$, or $\nu_i (x)$ for cases where no confusion can arise.

\begin{theorem}\label{main theorem o2 exp} In the $\bar{O}^{2}$- expansion of  almost all real numbers from the unit interval any digit $i$ from the  alphabet $A= \mathbb{N} $ appears  only finitely many times, i.e.,
for Lebesgue almost all real numbers $x \in [0,1]$ and for any  symbol  $i \in \mathbb{N}$ one has:
$$
\limsup_{n \to \infty} N_i(x,n) < +\infty.
$$
\end{theorem}
\begin{proof} Let  $N_i(x)$ be the number of symbols $"i"$ in the  $\bar{O}^{2}$-expansion of the real number  $x$. We shall prove that the Lebesgue measure of the set $$A_{i}=\{x:~N_{i}(x)=\infty\}$$ is equal to zero for any  $i\in N.$

To this end let us consider the set $$\bar{\Delta}_{i}^{k}=\{x: x=\bar{O}^{2}(d_{1},\ldots,d_{k-1},i,d_{k+1},\ldots); d_j \in N, \forall j\not=k\},$$
of all real numbers whose  $\bar{O}^{2}$-expansion contains the digit $i$ at the   $k$-th position, i.e.,  $d_{k}(x)=i.$
\begin{lemma}
For any  $i \in N$ and  $k \in N$ one has
 $$\lambda\left(\bar{\Delta}_{i}^{1}\right)=\frac{1}{i(i+1)}\leq\frac{1}{2},~~~$$
$$ \lambda(\bar{\Delta}_{i}^{k})  \leq \frac{1}{2^{2^{k-2}}} ~\mbox{for}~k>1.$$
\end{lemma}
\begin{proof}
Since $\bar{\Delta}_{i}^{1}=\Delta_{i}^{\bar{O}^{2}}=\left[\frac{1}{i+1};\frac{1}{i}\right],$ we have $\lambda\left(\bar{\Delta}_{i}^{1}\right)
=\frac{1}{i(i+1)}\leq \frac{1}{2}.$

From the above mentioned properties of cylindrical sets and from the definition of the set  $\bar{\Delta}_{i}^{k}$ it follows that
$$\bar{\Delta}_{i}^{k}=\bigcup\limits_{a_{1}=1}^{\infty}\ldots\bigcup\limits_{a_{k-1}=1}^{\infty}\Delta^{\bar{O}^{2}}_{a_{1}\ldots a_{k-1}i} $$
and $$\frac{|\Delta^{\bar{O}^{2}}_{a_{1}\ldots a_{k-1}i}|}{|\Delta^{\bar{O}^{2}}_{a_{1}\ldots a_{k-1}}|} \leq \frac{1}{2^{2^{k-2}}},$$

Therefore $$\lambda\left(\bar{\Delta}_{i}^{k}\right)=\sum\limits_{a_{1}=1}^{\infty}\ldots \sum\limits_{a_{k-1}=1}^{\infty}|\Delta^{\bar{O}^{2}}_{a_{1}\ldots a_{k-1}i}|\leq \frac{1}{2^{2^{k-2}}}\left(\sum\limits_{a_{1}=1}^{\infty}\ldots \sum\limits_{a_{k-1}=1}^{\infty}|\Delta^{\bar{O}^{2}}_{a_{1}\ldots a_{k-1}}|\right)=\frac{1}{2^{2^{k-2}}}.$$
\end{proof}

It is clear that  $A_{i}$ is the upper limit of the sequence of the set $\{\bar\Delta_{i}^{k}\},$ i.e., $$A_{i}=\limsup_{k\rightarrow\infty}\bar{\Delta}_{i}^{k}=
\bigcap_{m=1}^{\infty}\left(\bigcup_{k=m}^{\infty} \bar{\Delta}_{i}^{k} \right).$$
Taking into account the Borel-Cantelli Lemma and the fact that  $$\sum\limits_{k=1}^{\infty} \lambda(\bar{\Delta}_{i}^{k}) \leq \sum\limits_{k=1}^{\infty} \frac{1}{2^{2^{k-2}}} < +\infty,$$ we deduce that  $$\lambda(A_i)=0, \forall i \in N.  $$
So, $$\lambda(\bar{A_i})=1, \forall i \in N.  $$

Let $$\bar{A} = \bigcap_{i=1}^{\infty}  \bar{A_i}.$$ Then  $\lambda(\bar{A})=1,$ which proves the theorem.
\end{proof}

\begin{corollary}\label{chastota i =0}
For Lebesgue almost all real numbers $x$ from the unit interval and $ \forall i\in N: \nu_{i}(x)=0.$
\end{corollary}

\begin{corollary}
For any stochastic vector  $\overrightarrow{p}=(p_{1},p_{2},\ldots,p_{k},\ldots)$ the set  $$I_{\overrightarrow{p}}=\{x:~x=\bar{O}^{2}(d_{1}(x),\ldots,d_{k}(x),\ldots),~\nu_{i}(x)=p_{i}~~\forall i\in N\}$$
is of zero Lebesgue measure.
\end{corollary}

\begin{theorem}
  There are no probability measures which are simultaneously invariant and ergodic w.r.t the one-sided shift transformation $T$ on the  $\bar{O}^2$-expansion, and  absolutely continuous w.r.t. Lebesgue measure.
\end{theorem}
\begin{proof}
To prove the theorem ad absurdum, let us assume that there exists an absolutely continuous probability measure $\nu$, which is invariant and ergodic w.r.t. the above defined transformation   $T$. Then for $\nu$-almost all  $x \in [0,1]$ (and, so, for a set of positive Lebesgue measure) and for any function $\varphi \in L^1([0,1], \nu)$ the following equality holds:
  $$
 \lim_{n\to \infty} \frac{1}{n} \sum_{j=0}^{n-1} \varphi(T^j(x)) = \int_{0}^{1} \varphi(x) d(\nu(x))= \int_{0}^{1} \varphi(x) f_{\nu}(x) dx,
 $$
 where $f_{\nu}(x)$ is the density of the measure $\nu$.

  Let  $\varphi_{i} (x) = 1$, if $x\in \Delta^{\bar{O}^2}_{i}$; and $\varphi_i (x)= 0$ otherwise.

 Then the condition  $$\int_{0}^{1} \varphi_i(x) f_{\nu}(x) dx = \int_{x\in \Delta^{\bar{O}^2}_{i}} f_{\nu}(x) dx >0 ~~\mbox{holds at least for one digit} ~~i \in N.$$
 Let the latter condition hold for  some fixed number $i_{0}$.

 On the other hand we have
  $$\lim_{n \to \infty} \frac{1}{n} \sum_{j=0}^{n-1} \varphi_{i_0}(T^j(x)) = \lim_{n \to \infty}\frac{N_{i_0}(x,n)}{n} = 0$$
  for  $\lambda$-almost all $x \in [0,1]$.

 So,   $$\lim_{n \to \infty}\frac{N_{i_0}(x,n)}{n} = 0$$
  for  $\lambda$-almost all $x \in [0,1]$, and  simultaneously, we have
  $$\lim_{n \to \infty}\frac{N_{i_0}(x,n)}{n} > 0$$
  for a set of positive Lebesgue measure. This contradiction proves the theorem.\end{proof}

\begin{corollary}
The application of the classical "king approach"  to the development of the metric theory
  (i.e., the exploitation of the Birkgoff ergodic theorem with a measure which is absolutely continuous w.r.t. Lebesgue measure and which is ergodic and invariant w.r.t. the one-sided shift transformation on the corresponding expansion) is impossible for  the case of  the  $\bar{O}^2$-expansion for real numbers.
\end{corollary}

\section{On singularity of the random $ \bar{O}^{2}$-expansion}

 Let
 \begin{equation}\label{eta}
   \eta = \bar{O}^{2} (\eta_1,\eta_2,\dots,\eta_k,\dots),
 \end{equation}
be the random $\bar{O}^{2}$-expansion with independent identically distributed $\bar{O}^{2}$-symbols $\eta_k$, i.e. $\eta_k$ are i.i.d. random variables taking values  $1$, $2$,~$\dots$, $m$,~$\dots$ with probabilities $p_{1}$, $p_{2}$,~$\dots$, $p_{m}$,~$\dots$
correspondingly, with $ p_{m}\geq0, \quad
\sum\limits_{m=1}^\infty {p_{m}}=1.
$

The following theorem completely describes the Lebesgue decomposition of the distribution of  $\eta$.

\begin{theorem} The random variable  $\eta$ with independent identically distributed increments of the second Ostrogradsy expansion (i.e., the random variable defined by equality (\ref{eta}) is:

1) of pure atomic distribution with a unique atom (iff $p_i =1$ for some $i \in N$);

2) or it is singularly continuously distributed (in all other cases).
\end{theorem}
\begin{proof}
1) The first statement  of the theorem is evident and we added it only for the completeness of the Lebesgue classification.

2) To prove the second statement we firstly remark that the probability measure  $\mu_{\eta}$ is invariant and ergodic w.r.t.  $T$. So, from the latter Theorem it follows that $\mu_{\eta}$ can not be absolutely continuous.  Let us show now  that in the case of continuity the distribution of  $\eta$ can not contain an absolutely continuous component.

 Secondly let us remark that a direct application of the strong law of large numbers shows that for  $\mu_{\eta}$ - almost all  $x \in [0,1]$ the following equality holds
 \begin{equation}\label{chastota i_0 =p_i}
 \nu_{i}(x, \bar{O}^2) = p_i, ~~ \forall i\in N.
 \end{equation}

Let us choose a positive integer  $i_0$ such that $p_{i_0} >0$ (there exists at least one such a number) and let us consider the set
 $M_{i_0} = \{x: x \in [0,1], \nu_{i_0}(x, \bar{O}^2) = p_{i_0}>0 \}$. From (\ref{chastota i_0 =p_i}) it follows that this set is of full $\mu_{\eta}$-measure.

Let us also consider the set $L^{*}_{i_0} = \{x: x \in [0,1], ~ \nu_{i_0}(x, \bar{O}^2) = 0 \}$.  From the corollary after Theorem \ref{main theorem o2 exp} it follows that  $ \lambda(L^{*}_{i_0}) =1$. The sets $ M_{i_0}$ and  $L^{*}_{i_0}$  do not intersect. The first of them is a support of  the probability measure $\mu_{\eta}$, and the second one is the support of the Lebesgue measure on the unit interval. So, the measure $\mu_{\eta}$ is purely singular w.r.t. Lebesgue measure, which proves the theorem.
\end{proof}

\section{Cantor-like sets related to the  $\bar{O^{2}}$-expansion}

Let $ V_{k}\subset N $ and let
$$C=C[\bar{O^{2}},\{V_{k}\}]=\{x:~x=\bar{O}^{2}(d_{1},d_{2},\ldots,d_{k},\ldots),d_{k}\in V_{k},\forall k\in N\}.$$

The main goal of this section is to study metric and fractal properties of the set $C[\bar{O^{2}},\{V_{k}\}]$ and their dependence on the sequence of sets $\{V_k\}$.

It is clear that the set $C[\bar{O^{2}},\{V_{k}\}]$ is nowhere dense if and only if the condition $V_{k}\neq N$ holds for infinitely many indices  $k$.

Let $$F_{k}=\{x:~x\in \Delta_{a_{1}a_{2}\ldots a_{k}}^{\bar{O^{2}}},~a_{j}\in V_{j},j=\overline{1,k}\},$$
$$\overline{F}_{k+1}=\{x:~x\in \Delta_{a_{1}a_{2}\ldots a_{k}a_{k+1}}^{\bar{O^{2}}},~a_{j}\in V_{j},j=\overline{1,k}, a_{k+1}\not\in V_{k+1}\} $$

$$\lambda(\overline{F}
_{k+1})=\lambda(F_{k})-\lambda(F_{k+1})$$
$$\frac{\lambda(\overline{F}
_{k+1})}{\lambda(F_{k})}=1-\frac{\lambda(F_{k+1})}{\lambda(F_{k})}$$

\begin{lemma}
$$\lambda(C)=\prod\limits_{k=1}^{\infty}\left[1-\frac{\lambda(\overline{F}
_{k})}{\lambda(F_{k-1})}\right].$$
\end{lemma}
\begin{proof}
$$\lambda(C)=\lim\limits_{k\rightarrow\infty}\lambda(F_{k})=\lim\limits_{k\rightarrow\infty}\frac{\lambda(F_{k})}{\lambda(F_{k-1})}\cdot
\frac{\lambda(F_{k-1})}{\lambda(F_{k-2})}\cdot\ldots\cdot\frac{\lambda(F_{2})}{\lambda(F_{1})}\cdot\frac{\lambda(F_{1})}{\lambda(F_{0})}=$$
$$=\lim\limits_{k\rightarrow\infty}\prod\limits_{i=1}^{k}\frac{\lambda(F_{i})}{\lambda(F_{i-1})}=
\prod\limits_{k=1}^{\infty}\frac{\lambda(F_{k})}{\lambda(F_{k-1})}=\prod\limits_{k=1}^{\infty}\left[1-\frac{\lambda(\overline{F}
_{k})}{\lambda(F_{k-1})}\right].$$
\end{proof}

\begin{corollary}\label{kryterij dodatnosti miry lebega}
$$\lambda\left(C[\bar{O}^2,\{V_k\}]\right)>0~~\Leftrightarrow~~\sum\limits_{k=1}^{\infty}\frac{\lambda(\overline{F}_{k})}{\lambda(F_{k-1})}<+\infty;$$
$$\lambda\left(C[\bar{O}^2,\{V_k\}]\right)=0~~\Leftrightarrow~~\sum\limits_{k=1}^{\infty}\frac{\lambda(\overline{F}_{k})}{\lambda(F_{k-1})}=\infty.$$
\end{corollary}

a) First of all let us consider the case $$V_{k}=\{v_{k}+1,v_{k}+2,\ldots\},$$ where $\{v_{k}\}$ is an arbitrary sequence of positive integers.

\begin{theorem}\label{MnVk}
Let $V_k=\{v_{k}+l, l \in N \}$, where $v_{k}\in N$. If there exists a positive integer $b$ such that
$$\sum\limits_{k=1}^\infty \frac{v_{k+1}}{b^{2^{k}}} < +\infty,$$ then $\lambda (C[\bar{O}^2,\{V_{k}\}])>0$.
\end{theorem}
\begin{proof} Let $c$ be an arbitrary positive integer such that $c > max\{v_{1}, b\}$ and let us consider the cylinder $\Delta_{c}^{\bar{O}^{2}}$ of the first rank. Then
$$\lambda\left(\overline{F}_{k+1}\bigcap\Delta_{c}^{\bar{O}^{2}}\right)=\sum\limits_{a_{2}\in~
V_{2},~\ldots, ~a_{k}\in~
V_{k}}\sum\limits_{j=1}^{v_{k+1}}|\overline{\Delta}_{c a_{2}\ldots
a_{k}j}|.$$
Taking into account Lemma \ref{Lemma ocinka vidnoshennya cylindriv}, we have
$$\sum\limits_{j=1}^{v_{k+1}}|\overline{\Delta}_{c a_{2}\ldots
a_{k}j}|\leq v_{k+1}\cdot|\overline{\Delta}_{c a_{2}\ldots
a_{k}1}|<v_{k+1}\cdot\frac{1}{b^{2^{k}}}\cdot|\overline{\Delta}_{c a_{2}\ldots
a_{k}}|,$$ and therefore
$$\frac{\lambda\left(\overline{F}_{k+1}\bigcap\Delta_{c}^{\bar{O}^{2}}\right)}{\lambda\left({F_{k}}\bigcap\Delta_{c}^{\bar{O}^{2}}\right)}=
\frac{\sum\limits_{a_{2} \in~V_{2},~\ldots, ~a_{k}\in~
V_{k}}\sum\limits_{j=1}^{v_{k+1}}|\overline{\Delta}_{c a_{2}\ldots
a_{k}j}|}{\sum\limits_{a_{2} \in~ V_{2},~\ldots, ~a_{k}\in~
V_{k}}|\overline{\Delta}_{c a_{2}\ldots a_{k}}|}\leq$$
$$\leq \frac{\sum\limits_{a_{2} \in~ V_{2},~\ldots, ~a_{k}\in~
V_{k}}\frac{v_{k+1}}{b^{2^{k}}}|\overline{\Delta}_{c a_{2}\ldots
a_{k}}|}{\sum\limits_{a_{2} \in~ V_{2},~\ldots, ~a_{k}\in~
V_{k}}|\overline{\Delta}_{c a_{2}\ldots
a_{k}}|}=\frac{v_{k+1}}{b^{2^{k}}}.$$

From the construction of the set $C[\bar{O}^2,\{V_k\}]$ it follows that
$$\lambda(C[\bar{O}^2,\{V_k\}]) \geq \lambda(C[\bar{O}^2,\{V_k\}] \bigcap \Delta_{c}^{\bar{O}^{2}}) =  \lim\limits_{k\rightarrow\infty}\lambda(F_{k} \bigcap \Delta_{c}^{\bar{O}^{2}})=$$ $$=\lim\limits_{k\rightarrow\infty}\frac{\lambda(F_{k}\bigcap \Delta_{c}^{\bar{O}^{2}})}{\lambda(F_{k-1}\bigcap \Delta_{c}^{\bar{O}^{2}})}\cdot
\frac{\lambda(F_{k-1}\bigcap \Delta_{c}^{\bar{O}^{2}})}{\lambda(F_{k-2}\bigcap \Delta_{c}^{\bar{O}^{2}})}\cdot\ldots\cdot\frac{\lambda(F_{2}\bigcap \Delta_{c}^{\bar{O}^{2}})}{\lambda(F_{1}\bigcap \Delta_{c}^{\bar{O}^{2}})}\cdot\frac{\lambda(F_{1}\bigcap \Delta_{c}^{\bar{O}^{2}})}{\lambda(F_{0}\bigcap \Delta_{c}^{\bar{O}^{2}})}=$$
$$=
\prod\limits_{k=1}^{\infty}\frac{\lambda(F_{k}\bigcap \Delta_{c}^{\bar{O}^{2}})}{\lambda(F_{k-1}\bigcap \Delta_{c}^{\bar{O}^{2}})}=\prod\limits_{k=1}^{\infty}\left[1-\frac{\lambda(\overline{F}
_{k}\bigcap \Delta_{c}^{\bar{O}^{2}})}{\lambda(F_{k-1}\bigcap \Delta_{c}^{\bar{O}^{2}})}\right].$$

Since $$\sum_{k=1}^{\infty}\frac{\lambda\left(\overline{F}_{k+1}\bigcap\Delta_{c}^{\bar{O}^{2}}\right)}{\lambda\left({F_{k}}\bigcap\Delta_{c}^{\bar{O}^{2}}\right)}
\leq\sum_{k=1}^{\infty}\frac{v_{k+1}}{b^{2^{k}}}$$  and the series
$\sum\limits_{k=1}^{\infty}\frac{v_{k+1}}{b^{2^{k}}}$ converges, we have
$\lambda\left(C[\bar{O}^2,\{V_k\}]\right) \geq \lambda\left(C[\bar{O}^2,\{V_k\}]\bigcap\Delta_{c}^{\bar{O}^{2}}\right)>0,$ which proves the Theorem.
\end{proof}

\begin{proposition}\label{Proposition 1}
If $V_k =\{m+1, m+2, m+3, ...\}$ for some $m \in N$, then
\begin{equation}\label{Oc1MLO_1}
\lambda(C[\bar{O}^1,\{V_k\}])>0;
\end{equation}
\begin{equation}\label{Oc1lMLc.f.}
 \lambda(C[c.f.,\{V_k\}])=0;
\end{equation}
\begin{equation}\label{Oc1MLO_2}
\lambda(C[\bar{O}^2,\{V_k\}])>0.
\end{equation}
\end{proposition}
\begin{proof}
Inequality (\ref{Oc1MLO_1}) follows from Theorem 5 of the paper \cite{ABPT}.

To prove (\ref{Oc1lMLc.f.}) let us remind that for $\lambda$-almost all $x\in[0,1]$ a given digit "$i$" appears in the continued fraction expansion of the real number $x$ with the asymptotic frequency $\frac{1}{\ln 2}\cdot\ln\frac{(i+1)^{2}}{i(i+2)}$. This fact is actually a direct corollary from the Birkgoff ergodic theorem and the fact that the Gauss measure $G(E):=\frac{1}{\ln 2}\int\limits_{E} \frac{1}{1+x}dx$ is invariant and ergodic w.r.t. one-sided shift transformation on the continued fraction expansion.

If $i=1$, then \begin{equation}\label{frequency 1 for c.f.}
  \nu_{1}^{c.f.}(x)=\frac{1}{\ln 2} \ln\frac{4}{3}
\end{equation} for $\lambda$-almost all $x\in[0,1].$ On the other hand, $$\nu_{1}^{c.f.}(x)=0,~\forall x\in C[c.f., \{V_{k}\}],$$ which proves equality (\ref{Oc1lMLc.f.}).

Finally, inequality (\ref{Oc1MLO_2}) is a direct corollary of the latter theorem.
\end{proof}

\begin{proposition}\label{Proposition 2}
If $V_k =\{3^k+1, 3^k+2, 3^k+3, ...\}$, $k \in N$, then
\begin{equation}\label{Oc2MLO_1}
\lambda(C[\bar{O}^1,\{V_k\}])=0;
\end{equation}
\begin{equation}\label{Oc2lMLc.f.}
\lambda(C[c.f.,\{V_k\}])=0;
\end{equation}
\begin{equation}\label{Oc2MLO_2}
\lambda(C[\bar{O}^2,\{V_k\}])>0.
\end{equation}
\end{proposition}

\begin{proof}
Equality (\ref{Oc2lMLc.f.}) follows from the proof of the previous proposition and inequality (\ref{Oc2MLO_2}) is a  direct consequence of the latter theorem.

To prove (\ref{Oc2MLO_1}), let us prove two auxiliary lemmas, which have a self standing interest.
\begin{lemma}\label{Lemma1}
Let $\bar{O}_{g_{1}(x)g_{2}(x)\ldots g_{n}(x)\ldots}^{1}$ be the Ostrogradskyi-Pierce expansion in the difference form, $g_{n}(x)\in N.$

Then for $\lambda$-almost all $x\in[0,1]$ one has
\begin{equation}\label{VGr}
\overline{\lim_{n\rightarrow\infty}}\sqrt[n]{g_{n}(x)} \leq e
\end{equation}
\end{lemma}
\begin{proof}
In \cite{Sha86} it has been proven that for the standard Ostrogradskyi-Pierce expansion $\bar{O}_{q_{1}(x)q_{2}(x)\ldots q_{n}(x)\ldots}~(q_{n+1}>q_{n})$ for almost all (in the sense of Lebesgue measure) $x\in[0,1]$ one has:
\begin{equation}\label{Gr}
\lim_{n\rightarrow\infty}\sqrt[n]{q_{n}(x)}=e,
\end{equation}
where $q_{n}(x)=g_{1}(x)+g_{2}(x)+\ldots+g_{n}(x).$

From (\ref{Gr}) it follows that $\forall \varepsilon>0~\exists~N(\varepsilon,x)>0:~~\forall n>N(\varepsilon,x)$
$$(e-\varepsilon)^{n}<q_{n}(x)<(e+\varepsilon)^{n} $$
$$(e-\varepsilon)^{n+1}<q_{n+1}(x)<(e+\varepsilon)^{n+1} $$
So, $g_{n}(x)=q_{n+1}(x)-q_{n}(x)<(e+\varepsilon)^{n+1}-(e-\varepsilon)^{n}<(e+\varepsilon)^{n+1}.$
\\Therefore, $$\sqrt[n]{g_{n}(x)}<(e+\varepsilon)^{1+\frac{1}{n}},~\forall n>N(\varepsilon,x),$$
and, hence, $$\overline{\lim_{n\rightarrow\infty}}\sqrt[n]{g_{n}(x)}\leq e ,$$
which proves the lemma.
\end{proof}
\begin{lemma}\label{Lemma2}
If
\begin{equation}\label{MnV_k1}
V_{k}=\{v_{k}+l, l \in N\}
\end{equation}
and
\begin{equation}\label{NGr}
\underline{\lim}_{k\rightarrow\infty} \sqrt[k]{v_{k}}>e,
\end{equation}
then $$\lambda(C[\bar{O}^1,\{V_k\}])=0.$$
\end{lemma}
\begin{proof}
From (\ref{MnV_k1}) it follows that $g_{k}(x)>v_{k},~\forall x\in C[\bar{O}^{1},\{V_{k}\}].$
Therefore, $$\overline{\lim_{k\rightarrow\infty}}\sqrt[k]{g_{k}(x)}\geq \underline{\lim}_{k\rightarrow\infty}\sqrt[k]{v_{k}}>e,~\forall x\in C[\bar{O}^{1},\{V_{k}\}]. $$
Taking into account the latter lemma, we get
$$ \lambda\left(C[\bar{O}^{1},\{V_{k}\}]\right)=0.$$
\end{proof}
So, to show inequality (\ref{Oc2MLO_1}), it is enough to apply lemma \ref{Lemma2} with $v_{k}=3^{k}.$
\end{proof}

\textbf{Remark.} Proposition \ref{Proposition 1} shows essential differences between the metric theory of continued fractions and the metric theory of the $\bar{O}^{2}$-expansion. At the same time this proposition shows some level of similarity between $\bar{O}^{1}$- and $\bar{O}^{2}$-expansions of real numbers. Indeed, from  (\ref{Oc1MLO_1}) and (\ref{Oc1MLO_2}) it follows that the deleting of any finite number of digits from the alphabet does not affect the positivity of the Lebesgue measure of the set $C[\bar{O}^{1},\{V_{k}\}]$, and therefore, does not change the Hausdorff dimension of this set. Moreover, both the mapping $f_1:   E \rightarrow E \bigcap C[\bar{O}^{1},\{1,2,...,m\}] $ and the mapping $f_2:   E \rightarrow E \bigcap C[\bar{O}^{2},\{1,2,...,m\}] $ do not change the Hausdorff dimension of any subset $E \subset [0,1]$.

On the other hand Proposition \ref{Proposition 2} demonstrates obvious    differences in the metric theories of  $\bar{O}^{1}$- and $\bar{O}^{2}$-expansions. To stress these differences, let us mention that the transformations $T_{\bar{O}^{2} \rightarrow \bar{O}^{1}}$ and $T_{\bar{O}^{2} \rightarrow c.f.}$ are strictly increasing \textbf{singularly } continuous functions on the unit interval, where
\begin{equation} 
  T_{\bar{O}^{2} \rightarrow \bar{O}^{1}}(\Delta_{g_{1}(x)g_{2}(x)\ldots g_{k}(x)...}^{\bar{O^{2}}}) = \Delta_{g_{1}(x)g_{2}(x)\ldots g_{k}(x)...}^{\bar{O^{1}}},
\end{equation}
and
\begin{equation} 
  T_{\bar{O}^{2} \rightarrow c.f. }(\Delta_{g_{1}(x)g_{2}(x)\ldots g_{k}(x)...}^{\bar{O^{2}}}) = \Delta_{g_{1}(x)g_{2}(x)\ldots g_{k}(x)...}^{c.f.}.
\end{equation}
Actually the singularity of the function $T_{\bar{O}^{2} \rightarrow \bar{O}^{1}}$ follows from  (\ref{Oc2MLO_1}) and (\ref{Oc2MLO_2}). The singularity  of the function $T_{\bar{O}^{2} \rightarrow c.f.}$ follows from (\ref{frequency 1 for c.f.}) and from the fact that for $\lambda$-almost all $x \in [0,1]$ the asymptotic frequency of the digit $1$ in the $\bar{O}^{2}$-expansion of $x$ is equal to zero.

So, the continued fraction expansion, $\bar{O}^{1}$-expansion and $\bar{O}^{2}$-expansion are "mutually orthogonal", and their metric theories differ essentially.


b) Now let us consider the case where  \begin{equation}\label{V_k= 1,...,m_k}
  V_{k}=\{1,2,\ldots,m_{k}\}
\end{equation} and $\{m_k\}$ is an arbitrary sequence of positive integers.

Let $M_{1}:=1, $ and
$$M_{k+1}=(M_{k}+1)^{2}+m_{k+1},~\forall k\in N. $$
\begin{theorem}
If $\sum\limits_{k=1}^{\infty}\frac{M_{k}^{2}}{m_{k+1}}<+\infty,$
then  $\lambda\left(C[\bar{O}^{2},\{V_{k}\}]\right)>0.$
\end{theorem}
\begin{proof}
Let  $\Delta_{a_{1}\ldots a_{k}}^{\bar{O}^{2}}$ be an arbitrary cylinder of rank $k$ with $a_j \in V_j, \forall j \in \{1,2,...,k\}$. Then
$$\lambda\left(\overline{F}_{k+1}\bigcap\Delta_{a_{1}\ldots a_{k}}^{\bar{O}^{2}} \right)= \sum_{i=m_{k+1}+1}^{\infty} |\Delta_{a_{1}\ldots a_{k} i}^{\bar{O}^{2}}| =
\frac{1}{c_{k}(c_{k}+1)+m_{k+1}}, $$
where  $c_{k}$ can be calculated via $a_{1},a_{2},\ldots,a_{k}$ : $c_{i+1}= c_i (c_i+1)-1+a_{i+1}$ with $c_1=a_1$.
It is clear that
$$\lambda\left(F_{k}\bigcap\Delta_{a_{1}\ldots a_{k}}^{\bar{O}^{2}} \right)=\lambda\left(\Delta_{a_{1}\ldots a_{k}}^{\bar{O}^{2}}\right)=
\frac{1}{c_{k}(c_{k}+1)},~a_{j}\in V_{j},j=\overline{1,k}.$$
Therefore $$\frac{\lambda\left(\overline{F}_{k+1}\bigcap\Delta_{a_{1}\ldots a_{k}}^{\bar{O}^{2}} \right)}{\lambda\left(F_{k}\bigcap\Delta_{a_{1}\ldots a_{k}}^{\bar{O}^{2}} \right)}=
\frac{\frac{1}{c_{k}(c_{k}+1)+m_{k+1}}}{\frac{1}{c_{k}(c_{k}+1)}}=\frac{c_{k}(c_{k}+1)}{c_{k}(c_{k}+1)+m_{k+1}}.$$
From (\ref{V_k= 1,...,m_k}) it follows that $$1\leq c_{1}\leq m_{1},$$
$$2\leq c_{2}\leq c_{1}(c_{1}+1)-1+m_{2}\leq (m_{1}+1)^{2}+m_{2},$$
$$2\cdot 3 \leq c_{3}\leq c_{2}(c_{2}+1)-1+m_{3}\leq (c_{2}+1)^{2}+m_{3}=\left((m_{1}+1)^{2}+m_{2}\right)^{2}+m_{3},$$
$$\ldots\ldots\ldots\ldots\ldots\ldots\ldots\ldots\ldots\ldots\ldots\ldots $$
$$ {2^{2}}^{k-2}\leq c_{k}\leq (((m_{1}+1)^{2}+m_{2})^{2}+\ldots+m_{k-1})^{2}+m_{k}, ...$$
So,$${2^{2}}^{k-2}\leq c_{k}\leq M_{k}, \forall k \in N.$$
Therefore
 \begin{equation}\label{Oc}
 \frac{{2^{2}}^{k-2}({2^{2}}^{k-2}+1)}{{2^{2}}^{k-2}({2^{2}}^{k-2}+1)+m_{k+1}}\leq \frac{\lambda\left(\overline{F}_{k+1}\bigcap\Delta_{a_{1}\ldots a_{k}}^{\bar{O}^{2}} \right)}{\lambda\left(F_{k}\bigcap\Delta_{a_{1}\ldots a_{k}}^{\bar{O}^{2}} \right)}\leq \frac{M_{k}(M_{k}+1)}{M_{k}(M_{k}+1)+m_{k+1}}.
 \end{equation}
Since the estimation (\ref{Oc}) holds for any cylinder $\Delta_{a_{1}\ldots a_{k}}^{\bar{O}^{2}},~a_{j}\in V_{j}, j=\overline{1,k},$ we get
$$\frac{{2^{2}}^{k-1}}{{2^{2}}^{k-1}+m_{k+1}}\leq \frac{\lambda(\overline{F}_{k+1})}{\lambda(F_{k})}\leq \frac{(M_{k}+1)^{2}}{(M_{k}+1)^{2}+m_{k+1}}.$$
If $\sum\limits_{k=1}^{\infty}\frac{(M_{k}+1)^{2}}{(M_{k}+1)^{2}+m_{k+1}}<+\infty,$ then
$\sum\limits_{k=1}^{\infty}\frac{\lambda(\overline{F}_{k+1})}{\lambda(F_{k})}<+\infty $, and taking into account corollary \ref{kryterij dodatnosti miry lebega} we deduce that  $\lambda\left(C[\bar{O}^{2},V_{k}]\right)> 0.$
One can easily verify that $$\frac{1}{1+\frac{m_{k+1}}{M_{k}^{2}}}=\frac{M_{k}^{2}}{M_{k}^{2}+m_{k+1}}\leq \frac{(M_{k}+1)^{2}}{(M_{k}+1)^{2}+m_{k+1}}=\frac{\left(\frac{M_{k}+1}{M_{k}}\right)^{2}\cdot M_{k}^{2}}{\left(\frac{M_{k}+1}{M_{k}}\right)^{2}\cdot M_{k}^{2}+m_{k+1}}\leq $$
$$\leq \frac{4 M_{k}^{2}}{4 M_{k}^{2}+m_{k+1}}\leq \frac{1}{1+\frac{m_{k+1}}{4M_{k}^{2}}}\leq \frac{4M_{k}^{2}}{m_{k+1}}.$$
So, $$\sum_{k=1}^{\infty}\frac{(M_{k}+1)^{2}}{(M_{k}+1)^{2}+m_{k+1}}<+\infty~\Leftrightarrow~
\sum_{k=1}^{\infty}\frac{M_{k}^{2}}{m_{k+1}}<+\infty.$$
Therefore the convergence of the series $\sum\limits_{k=1}^{\infty}\frac{M_{k}^{2}}{m_{k+1}}$ implies the positivity of the Lebesgue measure of the set $C[\bar{O}^{2},\{V_{k}\}].$
\end{proof}
\begin{theorem}
If $\sum\limits_{k=1}^{\infty}\frac{{2^{2}}^{k-1}}{{2^{2}}^{k-1}+m_{k+1}}=+\infty,$ then $\lambda\left(C[\bar{O}^{2},\{V_{k}\}]\right)= 0.$
\end{theorem}
\begin{proof}
Using the estimation (\ref{Oc}) we have
$$\frac{\lambda\left(\overline{F}_{k+1}\bigcap\Delta_{a_{1}\ldots a_{k}}^{\bar{O}^{2}} \right)}{\lambda\left(F_{k}\bigcap\Delta_{a_{1}\ldots a_{k}}^{\bar{O}^{2}} \right)}\geq \frac{{2^{2}}^{k-1}}{{2^{2}}^{k-1}+m_{k+1}} $$
for all cylinders $\Delta_{a_{1}\ldots a_{k}}^{\bar{O}^{2}},~a_{j}\in V_{j}, j=\overline{1,k}.$

Therefore $$\frac{\lambda\left(\overline{F}_{k+1}\right)}{\lambda(F_{k})}\geq \frac{{2^{2}}^{k-1}}{{2^{2}}^{k-1}+m_{k+1}}, ~\forall k\in N.$$
If $\sum\limits_{k=1}^{\infty}\frac{{2^{2}}^{k-1}}{{2^{2}}^{k-1}+m_{k+1}}=+\infty,$ then, applying corollary \ref{kryterij dodatnosti miry lebega}, we get $\lambda\left(C[\bar{O}^{2},\{V_{k}\}]\right)= 0,$ which proves the Theorem
\end{proof}

\begin{proposition}\label{proposition 4}
  Let $V_k = \{1,2,..., 2^{2^{k-1}}\}.$
   Then \begin{equation}\label{Oc4MLO_1}
\lambda(C[\bar{O}^1,\{V_k\}])>0;
\end{equation}
\begin{equation}\label{Oc4lMLc.f.}
 \lambda(C[c.f.,\{V_k\}])>0;
\end{equation}
\begin{equation}\label{Oc4MLO_2}
\lambda(C[\bar{O}^2,\{V_k\}])=0.
\end{equation}
\end{proposition}
\begin{proof}
 Equality  (\ref{Oc4MLO_2}) is a direct corollary of the latter theorem.

To prove inequality  (\ref{Oc4MLO_1}), let us remind (see, e.g., \cite{BPT07}) that the condition $\sum\limits_{k=1}^{\infty} \frac{m_1+m_2+...+m_k}{m_{k+1}} < +\infty$ implies the positivity of the Lebesgue measure of the set $C[\bar{O}^1,\{V_k\}]$ with  $V_k = \{1,2,..., m_k\}.$  Since the series $ \sum\limits_{k=1}^{\infty} \frac{m_1+m_2+...+m_k}{m_{k+1}} $ diverges for $m_k = 2^{2^{k-1}}$, we get (\ref{Oc4MLO_1}).

 To prove inequality  (\ref{Oc4lMLc.f.}), let us remind (see, e.g., \cite{Khi61/63}) that $$\frac{1}{3i^{2}}\leq \frac{|\Delta_{a_{1}\ldots a_{n}i}^{c.f.}|}{|\Delta_{a_{1}\ldots a_{n}}^{c.f.}|}\leq \frac{1}{i^{2}}.$$
 So,   $$ \frac{\sum\limits_{i \not \in V_{k+1}}|\Delta_{a_{1}\ldots a_{n}i}^{c.f.}|}{|\Delta_{a_{1}\ldots a_{n}}^{c.f.}|}\leq \sum\limits_{i=2^{2^{k-1}}+1}^{\infty} \frac{2}{i^{2}} < \frac{4}{2^{2^{k-2}}}, $$
 and  therefore $$\sum\limits_{k=1}^{\infty} \frac{\lambda(\overline{F}^{c.f.}_{k+1})}{\lambda(F^{c.f.}_{k})} \leq  \sum\limits_{k=1}^{\infty} \frac{4}{2^{2^{k-2}}} < +\infty,$$ which implies the positivity of the Lebesgue measure of the set $\lambda(C[c.f.,\{V_k\}])$.
\end{proof}

c) Finally let us consider the  case where both the set $V_k$ and the set $\overline{V}_k := N \setminus V_k$ are infinite for any $k_in N.$

\begin{theorem}
 Let $ V_{k}=N \setminus \{b_{1}^{(k)}, b_{2}^{(k)},\ldots,b_{m}^{(k)},\ldots\},$  where $\{b_{m}^{(k)}\}_{m=1}^{\infty}$ is an increasing sequence of positive integers $\forall k\in N$, and let for any  $k\in N$ exist a positive integer
$d_{k}\in N$ such that
\begin{equation}\label{Ym1}
b_{n+1}^{k}-b_{n}^{k}\leq d_{k},~\forall n\in N.
\end{equation}
Then if
\begin{equation}\label{Ym2}
 \sum\limits_{k=1}^{\infty}\frac{1}{b_{1}^{(k)}\cdot d_{k}}=+\infty,
 \end{equation}
 then $\lambda\left(C[\bar{O}^{2},\{V_{k}\}]\right)=0.$
\end{theorem}
\begin{proof}
Let $\Delta_{a_{1}a_{2}\ldots a_{k-1}}^{\bar{O}_{2}}$ be an arbitrary cylinder of  rank  $k-1$ with
$a_{j}\in V_{j},~\forall j\in\{1,2,\ldots,k-1\}.$ Then
$$|\Delta_{a_{1}a_{2}\ldots a_{k-1}b_{1}^{(k)}}^{\bar{O}^{2}}|\geq |\Delta_{a_{1}a_{2}\ldots a_{k-1}(b_{1}^{(k)}+1)}^{\bar{O}^{2}}|\geq\ldots\geq
|\Delta_{a_{1}a_{2}\ldots a_{k-1}(b_{2}^{(k)}-1)}^{\bar{O}^{2}}| .$$
Therefore $$|\Delta_{a_{1}a_{2}\ldots a_{k-1}b_{1}^{(k)}}^{\bar{O}^{2}}|\geq\frac{1}{b_{2}^{(k)}-b_{1}^{(k)}-1}\cdot
\sum_{i=b_{1}^{(k)}+1}^{b_{2}^{(k)}-1}|\Delta_{a_{1}a_{2}\ldots a_{k-1}i}^{\bar{O}^{2}}| $$
On the other hand,
$$\sum\limits_{i=1}^{b_{1}^{(k)}-1}|\Delta_{a_{1}a_{2}\ldots a_{k-1}i}^{\bar{O}^{2}}|=
\sum\limits_{i=1}^{b_{1}^{(k)}-1}\frac{1}{(c_{k-1}(c_{k-1}+1)-1+i)(c_{k-1}(c_{k-1}+1)+i)}=$$
$$=\sum\limits_{i=1}^{b_{1}^{(k)}-1}\left(\frac{1}{c_{k-1}(c_{k-1}+1)-1+i}-\frac{1}{c_{k-1}(c_{k-1}+1)+i}\right)=$$
$$=\frac{1}{c_{k-1}(c_{k-1}+1)}-\frac{1}{c_{k-1}(c_{k-1}+1)+b_{1}^{(k)}-1}=
\frac{b_{1}^{(k)}-1}{c_{k-1}(c_{k-1}+1)(c_{k-1}(c_{k-1}+1)+b_{1}^{(k)}-1)}.$$
Hence 
$$\frac{\sum\limits_{i=1}^{b_{1}^{(k)}-1}|\Delta_{a_{1}a_{2}\ldots a_{k-1}i}^{\bar{O}^{2}}|}{|\Delta_{a_{1}a_{2}\ldots a_{k-1}(b_{1}^{(k)})}^{\bar{O}^{2}}|}=\frac{\frac{b_{1}^{(k)}-1}{c_{k-1}(c_{k-1}+1)\left(c_{k-1}
(c_{k-1}+1)+b_{1}^{(k)}-1\right)}}{\frac{1}{\left(c_{k-1}
(c_{k-1}+1)+b_{1}^{(k)}-1\right)\left(c_{k-1}
(c_{k-1}+1)+b_{1}^{(k)}\right)}}=$$
$$=\frac{c_{k-1}(c_{k-1}+1)+b_{1}^{(k)}}{c_{k-1}(c_{k-1}+1)}\cdot \left(b_{1}^{(k)}-1\right)=\left(1+\frac{b_{1}^{(k)}}{c_{k-1}(c_{k-1}+1)}\right)(b_{1}^{(k)}-1)\leq$$ $$\leq \left(1+\frac{b_{1}^{(k)}}{{2^{2}}^{k-2}}\right)\cdot b_{1}^{(k)}=:l_{k}.$$
So, \begin{equation}\label{OcDC}
\left\{
\begin{array}{lll}
|\Delta_{a_{1}a_{2}\ldots a_{k-1}b_{1}^{(k)}}^{\bar{O}^{2}}|\geq \frac{1}{l_{k}}\cdot \sum\limits_{i=1}^{b_{1}^{(k)}-1}|\Delta_{a_{1}a_{2}\ldots a_{k-1}i}^{\bar{O}^{2}}|;\\
|\Delta_{a_{1}a_{2}\ldots a_{k-1}b_{1}^{(k)}}^{\bar{O}^{2}}|\geq \frac{1}{b_{2}^{(k)}-b_{1}^{(k)}}\cdot \sum\limits_{i=b_{1}^{(k)}+1}^{b_{2}^{(k)}-1}|\Delta_{a_{1}a_{2}\ldots a_{k-1}i}^{\bar{O}^{2}}|.
\end{array}
\right.
\end{equation}
$$|\Delta_{a_{1}a_{2}\ldots a_{k-1}b_{2}^{(k)}}^{\bar{O}^{2}}|\geq \frac{1}{b_{3}^{(k)}-b_{2}^{(k)}}\cdot \sum\limits_{i=b_{2}^{(k)}+1}^{b_{3}^{(k)}-1}|\Delta_{a_{1}a_{2}\ldots a_{k-1}i}^{\bar{O}^{2}}|; $$
$$\ldots\ldots\ldots\ldots\ldots\ldots\ldots\ldots\ldots\ldots\ldots\ldots\ldots\ldots\ldots\ldots\ldots $$
$$|\Delta_{a_{1}a_{2}\ldots a_{k-1}b_{n}^{(k)}}^{\bar{O}^{2}}|\geq \frac{1}{b_{n+1}^{(k)}-b_{n}^{(k)}}\cdot \sum\limits_{i=b_{n}^{(k)}+1}^{b_{n+1}^{(k)}-1}|\Delta_{a_{1}a_{2}\ldots a_{k-1}i}^{\bar{O}^{2}}|\geq \frac{1}{d_{k}} \cdot \sum\limits_{i=b_{n}^{(k)}+1}^{b_{n+1}^{(k)}-1}|\Delta_{a_{1}a_{2}\ldots a_{k-1}i}^{\bar{O}^{2}}|. $$
From  (\ref{OcDC}) it follows that
$$|\Delta_{a_{1}a_{2}\ldots a_{k-1}b_{1}^{(k)}}^{\bar{O}^{2}}|\geq\frac{1}{2l_{k}}\cdot \sum\limits_{i=1}^{b_{1}^{(k)}-1}|\Delta_{a_{1}a_{2}\ldots a_{k-1}i}^{\bar{O}^{2}}|+\frac{1}{2d_{k}}\sum_{i=b_{1}^{(k)}+1}^{b_{2}^{(k)}-1}|\Delta_{a_{1}a_{2}\ldots a_{k-1}i}^{\bar{O}^{2}}|\geq $$ $$\geq \frac{1}{2l_{k}\cdot d_{k}}\sum_{i=1,~i \neq b_{1}^{(k)}}^{b_{2}^{(k)}-1}|\Delta_{a_{1}a_{2}\ldots a_{k-1}i}^{\bar{O}^{2}}|.
 $$
  Therefore, $$\sum_{i \not \in V_{k} }|\Delta_{a_{1}a_{2}\ldots a_{k-1}i}^{\bar{O}^{2}}|\geq \frac{1}{2l_{k}d_{k}}\sum_{i\in V_{k}}|\Delta_{a_{1}a_{2}\ldots a_{k-1}i}^{\bar{O}^{2}}|,$$ $$~\forall (a_{1},a_{2},\ldots,a_{k-1}),~a_{j}\in V_{j},j\in \{1,2,\ldots,k-1\}.$$
  Hence
  $$
  \left\{
\begin{array}{lll}
\lambda (\overline{F}_{k})\geq\frac{1}{2l_{k}d_{k}}\cdot \lambda(F_{k});\\
\lambda(\overline{F}_{k})+\lambda(F_{k})=\lambda(F_{k-1}).
\end{array}
\right.
   $$
 So, $$\frac{\lambda(\overline{F}_{k})}{\lambda(F_{k-1})}\geq\frac{1}{2l_{k}d_{k}+1}\geq\frac{1}{4l_{k}d_{k}}. $$
  If $\sum\limits_{k=1}^{\infty}\frac{1}{l_{k}\cdot d_{k}}=+\infty,$ then $\sum\limits_{k=1}^{\infty}\frac{\lambda(\overline{F}_{k})}{\lambda(F_{k-1})}=+\infty, $ and, therefore, $\lambda(C[\bar{O}^{2},\{V_{k}\}])=0.$
 It is clear that $$\sum\limits_{k=1}^{\infty}\frac{1}{l_{k}d_{k}} =\sum\limits_{k\in A}\frac{1}{\left(1+\frac{b_{1}^{(k)}}{{2^{2}}^{k-2}}\right)\cdot b_{1}^{(k)}\cdot d_{k}}+\sum\limits_{k \not\in A}
  \frac{1}{\left(1+\frac{b_{1}^{(k)}}{{2^{2}}^{k-2}}\right)\cdot b_{1}^{(k)}\cdot d_{k}},$$
  where $A=\left\{k:~\frac{b_{1}^{(k)}}{{2^{2}}^{k-2}}\leq 1\right\}.$

  The series $\sum\limits_{k\in A}\frac{1}{\left(1+\frac{b_{1}^{(k)}}{{2^{2}}^{k-2}}\right)\cdot b_{1}^{(k)}\cdot d_{k}}$ diverges if and only if the series  $\sum\limits_{k=1}^{\infty}\frac{1}{b_{1}^{(k)}\cdot d_{k}}$ does, and the series $$\sum\limits_{k \not\in A}
  \frac{1}{\left(1+\frac{b_{1}^{(k)}}{{2^{2}}^{k-2}}\right)\cdot b_{1}^{(k)}\cdot d_{k}}$$ always converges.

   Therefore the divergence of the series $ \sum\limits_{k=1}^{\infty}\frac{1}{l_{k}\cdot d_{k}}$ is equivalent to the divergence of the series $~\sum\limits_{k=1}^{\infty}\frac{1}{b_{1}^{(k)}\cdot d_{k}}$, which proves the Theorem.
\end{proof}
\begin{corollary}
 Let $V_{k}=N\setminus\{b_{1},b_{2},\ldots,b_{m},\ldots\}.$ If $$\exists~d\in N:~b_{n+1}-b_{n}\leq d~\forall n\in N,$$ then $\lambda\left(C[\bar{O}^{2},\{V_{k}\}]\right)=0.$
\end{corollary}
\begin{corollary}
If $V_{k}=N\setminus\{1,3,5,\ldots\},$   then $\lambda\left(C[\bar{O}^{2},\{V_{k}\}]\right)=0.$
\end{corollary}

\begin{theorem}
If $V_{k}=N\setminus\{1,4,9,\ldots,m^{2},\ldots\},$ then $\lambda\left(C[\bar{O}^{2},\{V_{k}\}]\right)>0.$
\end{theorem}
\begin{proof}
Let $\Delta_{a_{1}a_{2}\ldots a_{k-1}}^{\bar{O}_{2}}$ be an arbitrary cylinder of  rank  $k-1$ with
$a_{j}\in V_{j},~\forall j\in\{1,2,\ldots,k-1\}.$
Then
$$\lambda\left(\overline{F}_{k}\bigcap\Delta_{a_{1}\ldots a_{k-1}}^{\bar{O}^{2}} \right)=\sum\limits_{m=1}^{\infty}\frac{1}{\left(c_{k-1}(c_{k-1}+1)-1+m^{2}\right)
\left(c_{k-1}(c_{k-1}+1)+m^{2}\right)},$$

$$
|\Delta_{a_{1}\ldots a_{k-1}}^{\bar{O}^{2}}|=\frac{1}{c_{k-1}(c_{k-1}+1)}, 
.$$

$$\frac{\lambda\left(\overline{F}_{k}\bigcap\Delta_{a_{1}\ldots a_{k-1}}^{\bar{O}^{2}} \right)}{\lambda\left(\Delta_{a_{1}\ldots a_{k-1}}^{\bar{O}^{2}} \right)} = 
\sum\limits_{m=1}^{\infty}\frac{c_{k-1}(c_{k-1}+1)}{(c_{k-1}(c_{k-1}+1)-1+m^{2})(c_{k-1}(c_{k-1}+1)+m^{2})} <$$
$$
<\sum\limits_{m=1}^{\infty}\frac{c_{k-1}(c_{k-1}+1)-1+m^{2}}{(c_{k-1}(c_{k-1}+1)-1+m^{2})(c_{k-1}(c_{k-1}+1)+m^{2})}=
\sum\limits_{m=1}^{\infty}\frac{1}{c_{k-1}(c_{k-1}+1)+m^{2}}< $$ $$<\int\limits_{0}^{+\infty}\frac{1}{c_{k-1}(c_{k-1}+1)+x^{2}}dx=\frac{\pi}{2 \cdot \sqrt{(c_{k-1}(c_{k-1}+1))}}\leq \frac{\pi}{2\cdot 2^{{2^{k-2}}}}
$$


Therefore $$\frac{\lambda(\overline{F}_{k})}{\lambda(F_{k-1})}<\frac{\pi}{2\cdot 2^{{2^{k-2}}}}.$$

 Since $$\sum_{k=1}^{\infty}\frac{\lambda(\overline{F}_{k})}{\lambda(F_{k-1})}<+\infty, $$ we get $\lambda\left(C[\bar{O}^{2},V_{k}]\right)> 0,$which proves the theorem.
\end{proof}

\section{Fractal properties of real numbers with bounded $\bar{O}^2$-digits}

In the case of zero Lebesgue measure, the next level for the  study of properties of the sets $C[\bar{O}^2,\{V_k\}]$ is the determination
of their Hausdorff dimension  $\dim_H(\cdot)$ (see, e.g., \cite{Fal} for the definition and main properties of this main fractal dimension).


We shall study this problem for the case where $V_k = \{1,2,..., m_k\}$. A similar problem for the continued fraction expansion were studied by many authors during last  60 years.
Set
 $$E_2 = \{x: x = \Delta^{c.f.}_{\alpha_1(x)...\alpha_k(x)...~} ,\alpha_k(x) \in \{1,2\} \}.$$

In  1941   Good \cite{Good} shows that  $$0,5194 < \dim_H (E_2)
<0, 5433.$$
In 1982 and  1985 Bumby \cite{Bumby82, Bumby85} improves these bounds:
$$0,5312 < \dim_H (E_2) < 0,5314.$$
 In  1989 Hensley \cite{Hensley89}  shows that
 $$0,53128049 < \dim_H (E_2) < 0,53128051.$$
In 1996 the same author  (\cite{Hensley96}) improves his estimate up to $$0,5312805062772051416.$$
A new approach to the determination of the Hausdorff dimension of the set $E_2$  with a desired precision was developed by Jenkinson and Policott in 2001 \cite{Jenkinson}.

Our nearest aim is to study fractal properties of sets which are $\bar{O}^2$-analogues of the above discussed set $E_2$, i.e., the set
$C[\bar{O}^2,\{1,2\}]$ and its generalization $C[\bar{O}^2,\{1,2, ..., m\}]$.
The following theorem shows that from the fractal geometry point of view  the sets $E_2$ and $C[\bar{O}^2,\{1,2\}]$ (as well as their generalizations) are cardinally different.

\begin{theorem}\label{MnVk1}
Let  $V_k=\{1,2,3,..., m_{k}\}$, where $m_{k}\in N$.

 If for any positive $\alpha$ the following equality
 \begin{equation}\label{din_h = 0}
   \lim\limits_{k\to \infty} \frac{m_1 \cdot m_2\cdot ... \cdot m_k}{2^{(\alpha 2^{k-1})}}
=0
 \end{equation}
  holds, then the Hausdorff dimension of the set
$C[\bar{O}^2,\{V_k\}]$ is equal to zero, i.e.,
$$\dim_H (C[\bar{O}^2,\{V_k\}]) =0.$$
\end{theorem}
\begin{proof} From the construction of the sets $V_k$ it follows that the set $C[\bar{O}^2,\{V_k\}]$ can be covered by $m_{1}$ cylinders of the first rank, by  $m_{1}\cdot m_{2}$ cylinders of  rank 2, $\ldots,$ by
$m_{1}\cdot m_{2}\cdot\ldots\cdot m_{k}$ cylinders of rank $k$, ... .
The cylinder $\overline{\Delta}_{\underbrace{11\ldots
1}_{k}}$ has the maximal length among all cylinders of  rank  $k$:
$$|\overline{\Delta}_{\underbrace{11\ldots
1}_{k}}|\leq\frac{1}{c_{k}(c_{k}+1)}<\frac{1}{2^{(2^{k-1})}}.$$

Let us consider the coverings of the set $C[\bar{O}^2,\{V_k\}]$ by cylinders of  rank $k$:
$$\overline{\Delta}_{a_{1}a_{2}\ldots a_{k}},~~a_{1}\in
V_{1},a_{2}\in V_{2},\ldots,a_{k}\in V_{k}.$$ It is clear that
$\left(\bigcup\limits_{a_{1}\in V_{1},\ldots, a_{k}\in
V_{k}}\right)\supset C[\bar{O}^2,\{V_k\}]. $

For a given $\varepsilon>0$ there exists $k\in N$  such that
$\frac{1}{2^{(2^{k-1})}}<\varepsilon.$ In such a case the length of any  $k$-rank cylinder of the $\bar{O}^2$-expansion does not exceed $\varepsilon.$
So, $$H_{\varepsilon}^{\alpha}(C[\bar{O}^2,\{V_k\}])=\inf\limits_{|E_{i}|\leq\varepsilon,
(\bigcup E_{i})\supset C}\sum_{i}|E_{i}|^{\alpha}\leq \sum_{a_{1}\in
V_{1},\ldots,a_{k}\in V_{k}}|\overline{\Delta}_{a_{1}a_{2}\ldots
a_{k}}|^{\alpha}\leq $$ $$\leq m_{1}\cdot m_{2}\cdot \ldots \cdot
m_{k}\cdot\left(\frac{1}{2^{(2^{k-1})}}\right)^{\alpha}=\frac{m_{1}\cdot
m_{2}\cdot \ldots \cdot m_{k}}{2^{\alpha(2^{k-1})}}.$$
$$H^{\alpha}(C[\bar{O}^2,\{V_k\}])=\lim\limits_{\varepsilon\downarrow
0}H_{\varepsilon}^{\alpha}(C[\bar{O}^2,\{V_k\}])=\lim\limits_{k\rightarrow\infty}H_{\varepsilon_{k}}^{\alpha}(C[\bar{O}^2,\{V_k\}]),$$
where
$$\varepsilon_{k}=\frac{1}{2^{(2^{k-1})}}.$$
Since $$H_{\varepsilon_{k}}^{\alpha}(C[\bar{O}^2,\{V_k\}])\leq \frac{m_{1}\cdot
m_{2}\cdot \ldots \cdot m_{k}}{2^{\alpha(2^{k-1})}}, $$ we have
$$H^{\alpha}(C[\bar{O}^2,\{V_k\}])=\lim\limits_{k\rightarrow\infty}
H_{\varepsilon_{k}}^{\alpha}(C[\bar{O}^2,\{V_k\}])\leq
\lim\limits_{k\rightarrow\infty}\frac{m_{1}\cdot m_{2}\cdot \ldots
\cdot m_{k}}{2^{\alpha(2^{k-1})}}=0~~(\forall \alpha>0)$$.

Therefore, $H^{\alpha}(C[\bar{O}^2,\{V_k\}])=0,~~~\forall\alpha>0,$ and so
$$\dim_H(C[\bar{O}^2,\{V_k\}])=\inf\{\alpha:~~H^{\alpha}(C[\bar{O}^2,\{V_k\}])=0\} =0.$$
\end{proof}

\begin{corollary}
If there exists a number $a \in N$ such that $m_k \leq a^k,~~ \forall k \in
N$, then
 $$\dim_H (C[\bar{O}^2,\{V_k\}]) =0.$$
\end{corollary}

\begin{corollary}
 If  $m_k = m_0,~~ \forall k \in N$ for some positive integer $m_0$, then  $$\dim_H (C[\bar{O}^2,\{V_k\}]) =0.$$
\end{corollary}

Let $B(\bar{O}^2)$ be the set of all real numbers  with bounded $\bar{O}^2$-symbols, i.e.,
$$B(\bar{O}^{2})=\{x:~x=\Delta_{a_{1}(x)\ldots a_{k}(x)\ldots}:~\exists K(x)\in N:~ a_{j}(x)\leq  K(x),~\forall j \in N\}.$$

\begin{theorem}
The set $B(\bar{O}^2)$ of all numbers with bounded $\bar{O}^2$-symbols is an anomalously fractal set,
i.e., the Hausdorff dimension of $B(\bar{O}^2)$ is equal to $0$:
$$\dim_{H}B(\bar{O}^2)=0. $$
\end{theorem}
\begin{proof}
The set $B(\bar{O}^2)$ can be decomposed in the following way:
$$B(\bar{O}^2)=\bigcup_{i=1}^{\infty}B_{i}(\bar{O}^{2}),$$
where $$B_{i}(\bar{O}^{2})=\{x:~a_{j}(x)\leq i,~\forall j \in N \}.$$
Since, $$\dim_{H}B_i(\bar{O}^2)=0,~\forall i\in N, $$
we have $$\dim_{H}B(\bar{O}^2)=\sup_{i}\dim_{H}B_{i}(\bar{O}^{2})=0,$$
which proves the theorem.
\end{proof}

\begin{remark}
From \cite{KP66} it follows that the set of continued fractions with bounded partial quotients is of full Hausdorff dimension:
$$\dim_{H}(B(c.f.))=1, $$
which stresses the essential difference between the dimensional theories for the $\bar{O}^{2}$-expansion and the continued fraction  expansion.
\end{remark}

\textbf{ Acknowledgement } This work was partly supported by DFG 436 UKR
113/97 project, EU project STREVCOMS  and by the Alexander von Humboldt Foundation.

\end{document}